\documentclass[12pt]{amsart}
\usepackage{amssymb,bbold}
\usepackage[all]{xy}

\theoremstyle{plain}
\newtheorem{theorem}{Theorem}
\newtheorem{corollary}[theorem]{Corollary}
\newtheorem{lemma}[theorem]{Lemma}
\newtheorem{proposition}[theorem]{Proposition}

\theoremstyle{definition}

\newtheorem{example}[theorem]{Example}

\def\one{\mathbb 1}

\begin{document}
\baselineskip 18pt

\title[Unbounded absolute weak convergence in Banach lattices]
      {Unbounded absolute weak convergence in Banach lattices}

\author[O.~Zabeti]{Omid Zabeti}

\address[O.~Zabeti]
  {Department of Mathematics, Faculty of Mathematics,
   University of Sistan and Baluchestan, Zahedan,
   P.O. Box 98135-674. Iran}
\email{o.zabeti@gmail.com}


\begin{abstract}
Several recent papers investigated unbounded versions of order and
norm convergences in Banach lattices. In this paper, we study the unbounded variant of weak convergence and its relationship with other convergences. In particular, we characterize order continuous Banach lattices and reflexive Banach lattices in
terms of this convergence.
\end{abstract}

\date{\today}

\maketitle
Throughout this paper, $E$ stands for a Banach lattice. A net $(x_{\alpha})$ in $E$ is said to
be {\bf unbounded order convergent} ($uo$-convergent) to $x$ if for every $u\in E_{+}$ the net
$(|x_{\alpha}-x|\wedge u)$
converges to zero in order. It is called {\bf unbounded norm convergent}
($un$-convergent) to $x$ if
$\||x_{\alpha}-x|\wedge u\|\rightarrow 0$ for every $u\in E_{+}$. These concepts were
investigated in [DOT17, GX14, G14, GTX17, KMT17]. We consider the unbounded
version of weak convergence. We say that $(x_{\alpha})$ is {\bf unbounded absolutely weakly
convergent} ($uaw$-convergent) to $x$ if $(|x_{\alpha}-x|\wedge u)$
converges to zero weakly for
every $u\in E_{+}$; we write $x_{\alpha}\xrightarrow{uaw}x$. For undefined terminology, we refer the reader to
[AB06, GTX17].
\begin{lemma}\label{101}
\begin{itemize}

\item[\em (i)]{$uaw$-limits are unique};
\item[\em (ii)]{If $x_{\alpha}\xrightarrow{uaw}x$ and $y_{\beta}\xrightarrow{uaw}y$, then $a x_{\alpha} + b y_{\beta} \xrightarrow{uaw} a x+b y$, for any scalars $a,b$};
 \item[\em (iii)]{If $x_{\alpha}\xrightarrow{uaw}x$, then  $y_{\beta}\xrightarrow{uaw}x$, for every subnet $(y_{\beta})$ of $(x_{\alpha})$};
 \item[\em (iv)]{If $x_{\alpha}\xrightarrow{uaw}x$, then    $|x_{\alpha}|\xrightarrow{uaw}|x|$};
  \item[\em (v)]{$x_{\alpha}\xrightarrow{uaw}x$ iff    $(x_{\alpha}-x)\xrightarrow{uaw}0$}.
\end{itemize}
\end{lemma}
\begin{proof}
$(i)$ Suppose that $x_{\alpha}\xrightarrow{uaw}x$ and $x_{\alpha}\xrightarrow{uaw}y$. Let $u\in E_{+}$. It follows from
$|x-y|\wedge u\leq |x_{\alpha}-x|\wedge u+|x_{\alpha}-y|\wedge u$ that $f(|x-y|\wedge u)=0$ for every $f\in E^{*}_{+}$
and, therefore, for every $f\in E^*$. It follows that $|x-y|\wedge u=0$. Taking $u=|x-y|$, we
conclude that $|x-y|=0$, hence $x = y$.

The implications $(ii)-(v)$ are straightforward.
\end{proof}
It is clear that every absolutely weakly convergent net is $uaw$-convergent; the con-
verse is true for order bounded nets. The following example shows that in general the
two convergences differ: let $E=c_0$ and $x_n = n^2 e_n$, where $(e_n)$ stands for the standard basis of $c_0$. It is easy to see that $x_n
\xrightarrow {uaw}0$,
yet it is not absolutely weakly null.

It was shown in [GTX17] that every disjoint net is $uo$-null. Let $(x_{\alpha})$ be a
disjoint net. For every $u\in E_{+}$, the net
$(|x_{\alpha}|\wedge u)$
is disjoint and order bounded,
hence weakly null. This yields the following.
\begin{lemma}\label{100}
Every disjoint net is $uaw$-null.
\end{lemma}
The next result is analogous to [DOT17, Lemma 2.11]; the proof is similar.
\begin{lemma}
Let $e\in E_{+}$ be a quasi interior point. Then $x_{\alpha}\xrightarrow{uaw}0$ iff $|x_{\alpha}|\wedge e\xrightarrow{w}0$.
\end{lemma}
It was observed in Section 7 of [DOT17] that $un$-convergence is given by a topology,
and sets of the form $V_{u,\varepsilon}=\{x\in E, \||x|\wedge u\|<\varepsilon\}$; where $u\in E_{+}$ and $\varepsilon> 0$, form
a base of zero neighborhoods for this topology. Using a similar argument, one can
show that sets of the form
\[V_{u,\varepsilon,f}=\{x\in E: f(|x|\wedge u)<\varepsilon \},\]
where $u\in E_{+}$, $\varepsilon>0$, and $f\in E^{*}_{+}$ form a base of zero neighborhoods for a Haus-
dorff topology, and the convergence in this topology is exactly the $uaw$-convergence. Similarly to Lemmas 2.1 and 2.2 of [KMT17], for every $u\in E_{+}$, $\varepsilon>0$ , and $f \in E^{*}_{+}$,
either $V_{u, \varepsilon, f}$ is contained in $[-u,u]$, in which case $u$ is a strong unit, or $V_{u,\varepsilon, f}$ contains
a non-trivial ideal. Now, a natural conjecture can arise as follows:

{\bf Question}. Suppose $E$ is infinite-dimensional. Does every $uaw$ neighbourhood of zero contain a non-trivial ideal?

Clearly, $un$-convergence implies $uaw$-convergence. The converse is false in general;
the standard unit sequence $(e_n)$ in $\ell_{\infty}$ is $uaw$-null but not $un$-null.
\begin{theorem}\label{4}
The following are equivalent:
\begin{itemize}
\item[\em (i)]{$E$ is order continuous};
\item[\em (ii)]{$x_{\alpha}\xrightarrow{uaw}0$  $\Leftrightarrow$ $x_{\alpha}\xrightarrow{un}0$ for every net $(x_{\alpha})\subseteq E$};
\item[\em (iii)]{$x_n\xrightarrow{uaw}0$  $\Leftrightarrow$ $x_n\xrightarrow{un}0$ for every sequence $(x_n)\subseteq E$}.
\end{itemize}
\end{theorem}
\begin{proof}
$(i)$ implies $(ii)$ by [AB06, Theorem 4.17]. The implication $(ii)\Rightarrow(iii)$ is trivial.
To show that $(iii)\Rightarrow(i)$, let $(x_n)$ be a disjoint order bounded sequence. By Lemma \ref{100},
it is $uaw$-null. By assumption, it is $un$-null. Since it is order bounded, it is norm null.
It follows that E is order continuous.
\end{proof}
Note that $uo$-convergence does not imply $uaw$-convergence, in general: consider $E=C([0,1])$. Define the sequence $(f_n)\subseteq E$ via $f_n(0)=1$, $f_n(\frac{1}{n})=f_n(1)=0$, and linear in between. Then $(f_n)$ is $uo$-null but not $uaw$-null.
The following result is analogous to Theorem 2.1 in [G14]; the proof is similar (cf.
also Theorem 8.1 in [KMT17]). Note that sequences may be replaced by nets.
\begin{proposition}\label{5}
The following are equivalent:
\begin{itemize}
\item[\em (i)]{$E$ is order continuous};
\item[\em (ii)]{for every norm bounded sequence $({x}^{*}_n)\subseteq E^{*}$,  ${x}^{*}_n\xrightarrow{uaw}0$ implies that ${x}^{*}_n\xrightarrow{w^*}0$};
\item[\em (iii)]{for every norm bounded sequence $({x}^{*}_n)\subseteq E^{*}$, ${x}^{*}_n\xrightarrow{uaw}0$ implies that ${x}^{*}_n\xrightarrow{|\sigma|(E^{*},E)}0$}.
\end{itemize}
\end{proposition}
\begin{corollary}
Suppose $E$ is an order continuous Banach lattice. Then every norm bounded $uaw$-Cauchy net in $E^{*}$ is $w^*$-convergent.
\end{corollary}
\begin{proof}
Suppose $(x^{*}_{\alpha})$ is a norm bounded $uaw$-Cauchy net in $E^{*}$. By Proposition \ref{5}, $({x}^{*}_{\alpha})$ is $w^*$-Cauchy, hence $w^*$-convergent by Alaoglu's Theorem.
\end{proof}
Now, we characterize order continuity of the dual of a Banach lattice in term of $uaw$-convergence.
\begin{theorem}\label{13}
The following are equivalent:
\begin{itemize}
\item[\em (i)]{$E^*$ is order continuous};
\item[\em (ii)]{For every norm bounded net $(x_{\alpha})\subseteq E$, $x_{\alpha}\xrightarrow{uaw}0$ implies $x_{\alpha}\xrightarrow{w}0$};
\item[\em (iii)]{For every norm bounded sequence $(x_n)\subseteq E$, $x_n\xrightarrow{uaw}0$ implies $x_n\xrightarrow{w}0$}.
\end{itemize}
\end{theorem}
\begin{proof}
The proof of $(i)\Rightarrow(ii)$ is similar to that of [DOT17, Theorem 6.4]. $(ii)\Rightarrow(iii)$ is
trivial. To show $(iii)\Rightarrow(i)$, observe that every disjoint norm bounded sequence in $E$
is $uaw$-null by Lemma \ref{100}, hence weakly null. Now apply [AB06, Theorem 4.69].
\end{proof}
In the following results, we characterize reflexive Banach lattices in terms of unbounded convergences.
\begin{theorem}\label{14}
The following are equivalent:
\begin{itemize}
\item[\em (i)]{$E$ is reflexive};
\item[\em (ii)]{Every norm bounded $uaw$-Cauchy net in $E$ is weakly convergent};
\item[\em (iii)]{Every norm bounded $uaw$-Cauchy sequence in $E$ is weakly convergent}.
\end{itemize}
\end{theorem}
\begin{proof}
$(i)\Rightarrow(ii)$ Let $(x_{\alpha})$ be a bounded $uaw$-Cauchy net in $E$. Then $(x_{\alpha})$ is weakly
Cauchy by Theorem \ref{13}, hence weakly convergent by Alaoglu's Theorem.

$(ii)\Rightarrow(iii)$ Trivially.

$(iii)\Rightarrow(i)$ It suffices to show that $E$ contains no lattice copies of $c_0$ or $\ell_1$. Suppose
that $E$ contains a lattice copy of $\ell_1$. The unit vector basis of $\ell_1$ is disjoint and,
therefore, $uaw$-null in E by Lemma \ref{100}. By assumption, it is weakly convergent in $E$
and, therefore, in $\ell_1$; a contradiction.
Suppose now that $c_0$ embeds into $E$ as a sublattice. Let $(e_i)$ be the standard basis
of $c_0$; put $x_n=\sum_{i=1}^{n}e_i$. Fix $x^{*}\in E^{*}_{+}$
and $u\in E_{+}$. Since $(x_n)$ is weakly Cauchy in $c_0$,
we have $x^{*}(|x_m-x_n|\wedge u)\leq x^{*}(x_m-x_n)\rightarrow 0$ as $m,n\rightarrow \infty$ with $n<m$. Hence, $(x_n)$
is $uaw$-Cauchy, therefore, weakly convergent in $E$, hence in $c_0$; a contradiction.

\end{proof}
Combining this with Theorem \ref{4}, we obtain the following.
\begin{corollary}
For an order continuous Banach lattice $E$, the following are equivalent:
\begin{itemize}
\item[\em (i)]{$E$ is reflexive};
\item[\em (ii)]{ Every norm bounded $un$-Cauchy net in $E$ is weakly convergent};
\item[\em (iii)]{Every norm bounded $un$-Cauchy sequence in $E$ is weakly convergent}.
\end{itemize}
\end{corollary}
The order continuity assumption cannot be dropped: in $\ell_{\infty}$, every norm bounded
$un$-Cauchy net is norm Cauchy and, therefore, weakly convergent.

Recall that a net $(x_{\alpha})$ in a Banach lattice $E$ is $uo$-Cauchy if the net $(x_{\alpha}-x_{\beta})_{(\alpha,\beta)}$ is $uo$-null.
 \begin{theorem}
 The following are equivalent:
\begin{itemize}
\item[\em (i)]{$E$ is reflexive};
\item[\em (ii)]{Every norm bounded $uo$-Cauchy net in $E$ is weakly convergent};
\item[\em (iii)]{Every norm bounded $uo$-Cauchy sequence in $E$ is weakly convergent}.
\end{itemize}
 \end{theorem}
 \begin{proof}
The implication $(i)\Rightarrow(ii)$ follows from Theorem \ref{14} because in an order contin-
uous Banach lattice every $uo$-Cauchy net is $uaw$-Cauchy. $(ii)\Rightarrow(iii)$ Trivially. The
proof that $(iii)\Rightarrow (i)$ is similar to that in Theorem \ref{14}. If $E$ contains a lattice copy of $\ell_1$
then the standard basis of $\ell_1$ is disjoint, hence $uo$-null, so weakly null in $E$ and,
therefore, in $\ell_1$; a contradiction. Suppose that $c_0$ embeds into $E$ as a sublattice; let
$(e_i)$ be the standard basis of $c_0$; put $x_n =\sum_{i=1}^{n} e_i$. Then $(x_n)$ is $uo$-Cauchy in $c_0$.
It follows from [GTX17, Corollary 3.3] that $(x_n)$ is $uo$-Cauchy in $E$. Hence, $(x_n)$ is
weakly convergent in $E$ and, therefore, in $c_0$; a contradiction.
 \end{proof}
[G14, Theorem 3.4] shows that every $w^*$-null net in $E^*$ is $uo$-null iff $E$ is order
continuous and atomic. [KMT17, Theorem 8.4] asserts that every $w^*$-null net in $E^*$
is $un$-null iff $E^*$ is atomic and both $E$ and $E^*$ are order continuous. We show next
that the latter result remains valid if "$un$-null" is replaced with "$uaw$-null".
\begin{proposition}
The following are equivalent:
\begin{itemize}
\item[\em (i)]{ Every $w^*$-null net in $E^*$ is $uaw$-null};
\item[\em (ii)]{$E^*$ is atomic and both $E$ and $E^*$ are order continuous}.
\end{itemize}
\end{proposition}
\begin{proof}
$(ii)\Rightarrow(i)$ by [KMT17, Theorem 8.4]. Assume $(i)$ and suppose that $x^{*}_{\alpha}\downarrow 0$ in
$E^*$. It follows that $x^{*}_{\alpha}(x)\downarrow 0$ for each $x\in E_{+}$, so that $x^{*}_{\alpha}\xrightarrow{w^*}0$. By assumption, $x^{*}_{\alpha}\xrightarrow{uaw}0$, hence $x^{*}_{\alpha}\xrightarrow{w}0$ and, furthermore, $\|x^{*}_{\alpha}\|\rightarrow 0$ by Dini's Theorem [AB06,
Theorem 3.52]. It follows that $E^*$ is order continuous. By Theorem \ref{4}, every $uaw$-null
net in $E^*$ is $un$-null. Now apply [KMT17, Theorem 8.4].
\end{proof}
\begin{example}
Let $E =\ell_1$. $E$ is order continuous, $E^* = \ell_{\infty}$ is atomic but not order
continuous. Define $(x^{*}_{n})$ in $\ell_{\infty}$ via $x^{*}_{n} = (0,\ldots,0,1,\ldots)$, with $n$ zero terms. Then
$x^{*}_{n}\xrightarrow{w^*}0$, yet it is not weakly null, hence not $uaw$-null because it is contained in $[0, \one]$.
\end{example}

The following is a $uaw$ variant of [KMT17, Proposition 6.6]; the proof follows
easily from Theorem \ref{13}.
\begin{proposition}\label{34}
Suppose that $E^*$ is order continuous. Then every norm closed
convex norm bounded subset $C$ of $E$ is $uaw$-closed.
\end{proposition}
It is proved in [KMT17, Theorem 7.5] that the unit ball $B_E$ is $un$-compact iff $E$ is
an atomic $KB$-space.

\begin{proposition}\label{36}
$B_E$ is $uaw$-compact iff $E$ is an atomic $KB$-space.
\end{proposition}
\begin{proof}
It suffices to show that $uaw$-compactness of $B_E$ implies that $E$ is order contin-
uous; the result then follows from Theorem \ref{4} and [KMT17, Theorem 7.5]. Suppose
that $B_E$ is $uaw$-compact. Since order intervals are $uaw$-closed by Lemma \ref{101}, they are
$uaw$-compact, hence absolutely weakly compact, and, therefore, weakly compact. It
follows that $E$ is order continuous.
\end{proof}

{\bf Further remarks}. A preliminary version of this paper was posted on arXiv on
August 6, 2016. In particular, it was proved there that in order continuous Banach
lattices, $uaw$-convergence (and, therefore, $un$-convergence) is stable under passing to
and from a sublattice; this was independently proved in [KMT17] as Corollary 4.6.
Recently, there have been some results regarding unbounded convergence in locally solid vector lattices; see [DEM17, T17] for more details.

{\bf Acknowledgements}.
This note would not have existed without inspiring and worthwhile suggestions of V. G. Troitsky, my friend and my colleague. Thanks is also due to Niushan Gao and Foivos Xanthos for useful remarks. I would like to have a deep gratitude toward the referee for  invaluable comments which improved the paper in the present form.


\begin{thebibliography}{1}
\bibitem[AB06]{AB:06} C.D. Aliprantis and O. Burkinshaw, Positive operators, 2nd edition,
Springer, 2006.
\bibitem[DEM17]{DEM:17} Y. A. Dabboorasad, E. Y. Emelyanov, and M. A. A. Marabeh, {\em $u\tau$-convergence in locally solid vector lattices,} preprint, arXiv: 1706.02006v3 [math.FA].
\bibitem[DOT17]{Den:17} Y. Deng, M O'Brien, and V. G. Troitsky, {\em Unbounded norm convergence in Banach lattices,} Positivity, 21(3), (2017), pp 963--974.
\bibitem[G14]{G:14} N. Gao, {\em Unbounded order convergence in dual spaces,} J. Math. Anal. Appl., 419(2014), pp 347–-354.
\bibitem[GTX17]{GTX:17} N. Gao, V. G. Troitsky, and F. Xanthos, {\em Uo-convergence and its applications to Cesàro means in Banach lattices,}  Israel J. Math., 220 (2017), pp 649--689.
\bibitem[GX14]{GX:14} N. Gao and F. Xanthos, {\em Unbounded order convergence and application
to martingales without probability,} J. Math. Anal. Appl., 415 (2014), pp
931--947.
\bibitem[KMT17]{Kmt:17} M. Kandi\'{c}, M.A.A. Marabeh, and  V. G. Troitsky, {\em Unbounded norm topology in Banach lattices,} J. Math. Anal. Appl., 451 (2017), no. 1, pp 259--279.
\bibitem[T17]{T:17} M. A. Taylor, {\em Unbounded topologies and $uo$-convergence in locally solid vector lattices,} preprint, arXiv: 1706.01575v1 [math.FA].

\end{thebibliography}
\end{document}